\documentclass[11pt]{amsart}

\usepackage{graphics}
\usepackage{color}
\usepackage[a4paper, margin=3cm]{geometry}
\usepackage{array}
\usepackage{amssymb,enumerate}
\usepackage{amsmath}
\usepackage{bbm}           
\usepackage{bm}
\usepackage{eso-pic,graphicx}
\usepackage{tikz}
\usepackage{cite}
\usepackage{esint}
\usepackage[pdftex, bookmarksnumbered, bookmarksopen, colorlinks, citecolor=blue, linkcolor=blue]{hyperref}

\usepackage{graphicx}
\usepackage{bbding}
\usepackage{makecell}
\usepackage{amsmath,amsfonts}
\usepackage{rotating}
\usepackage{amssymb}
\usepackage{verbatim}
\usepackage{rotating}
\usepackage{mathrsfs}
\DeclareSymbolFont{largesymbols}{OMX}{cmex}{m}{n}
\makeatletter
\def\Ddots{\mathinner{\mkern1mu\raise\p@
\vbox{\kern7\p@\hbox{.}}\mkern2mu
\raise4\p@\hbox{.}\mkern2mu\raise7\p@\hbox{.}\mkern1mu}}
\makeatother

\def\XXint#1#2#3{{\setbox0=\hbox{$#1{#2#3}{\int}$}
\vcenter{\hbox{$#2#3$}}\kern-.5\wd0}}

\def\dashint{\Xint-}

\begin{document}

\newtheorem{hyp}{Hypothesis}

\newtheorem{hyp2}[hyp]{Hypothesis}

\newtheorem{definition}{Definition}
\newtheorem{theorem}[definition]{Theorem}
\newtheorem{proposition}[definition]{Proposition}
\newtheorem{conjecture}[definition]{Conjecture}
\def\theconjecture{\unskip}
\newtheorem{corollary}[definition]{Corollary}
\newtheorem{lemma}[definition]{Lemma}
\newtheorem{claim}[definition]{Claim}
\newtheorem{sublemma}[definition]{Sublemma}
\newtheorem{observation}[definition]{Observation}
\theoremstyle{definition}

\newtheorem{notation}[definition]{Notation}
\newtheorem{remark}[definition]{Remark}
\newtheorem{question}[definition]{Question}

\newtheorem{example}[definition]{Example}
\newtheorem{problem}[definition]{Problem}
\newtheorem{exercise}[definition]{Exercise}
 \newtheorem{thm}{Theorem}
 \newtheorem{cor}[thm]{Corollary}
 \newtheorem{lem}{Lemma}[section]
 \newtheorem{prop}[thm]{Proposition}
 \theoremstyle{definition}
 \newtheorem{dfn}[thm]{Definition}
 \theoremstyle{remark}
 \newtheorem{rem}{Remark}
 \newtheorem{ex}{Example}
 \numberwithin{equation}{section}

\def\N{\Bbb N}
\def\Z{\Bbb Z}
\def\R{\Bbb R}
\def\C{\Bbb C}
\def\hs{\hspace{0.25cm}}
\def\rr{{\mathbb R}}
\def\rn{{\mathbb{R}^n}}
\def\zz{{\mathbb Z}}
\def\nn{{\mathbb N}}
\def\cm{{\mathcal M}}
\def\ct{{\mathcal T}}
\def\fz{\infty }
\def\az{\alpha}
\def\bz{\beta}
\def\dz{\delta}
\def\ez{\epsilon}
\def\gz{{\gamma}}
\def\lz{\lambda}
\def\oz{{\omega}}
\def\boz{{\Omega}}
\def\tz{\theta}
\def\lf{\left}
\def\r{\right}
\def\N{\Bbb N}
\def\Z{\Bbb Z}
\def\R{\Bbb R}
\def\C{\Bbb C}
\def\rr{{\mathbb R}}
\def\rn{{\mathbb{R}^n}}
\def\zz{{\mathbb Z}}
\def\nn{{\mathbb N}}
\def\cm{{\mathcal M}}
\def\ct{{\mathcal T}}
\def\fz{\infty }
\def\az{\alpha}
\def\bz{\beta}
\def\dz{\delta}
\def\ez{\epsilon}
\def\gz{{\gamma}}
\def\lz{\lambda}
\def\oz{{\omega}}
\def\boz{{\Omega}}
\def\tz{\theta}
\def\lf{\left}
\def\r{\right}
\def\hs{\hspace{0.25cm}}

\def\hs{\hspace{0.25cm}}
\def\rr{{\mathbb R}}
\def\rn{{\mathbb{R}^n}}
\def\zz{{\mathbb Z}}
\def\nn{{\mathbb N}}
\def\cm{{\mathcal M}}
\def\ct{{\mathcal T}}
\def\fz{\infty }
\def\az{\alpha}
\def\bz{\beta}
\def\dz{\delta}
\def\ez{\epsilon}
\def\gz{{\gamma}}
\def\lz{\lambda}
\def\oz{{\omega}}
\def\boz{{\Omega}}
\def\tz{\theta}
\def\lf{\left}
\def\r{\right}
\def\hs{\hspace{0.25cm}}
\def\BMO{\mathop\mathrm{\,BMO\,}}
\def\Lip{\mathop\mathrm{\,Lip\,}}
\def\loc{{\mathop\mathrm{\,loc\,}}}
\def\dashint{\fint}
\def\ca{\mathcal{A}}
\def\Lip{\mathop\mathrm{\,Lip\,}}
\def\C{\mathbb{C}}
\def\R{\mathbb{R}}
\def\Rn{{\mathbb{R}^n}}
\def\Rns{{\mathbb{R}^{n+1}}}
\def\Sn{{{S}^{n-1}}}
\def\M{\mathbb{M}}
\def\N{\mathbb{N}}
\def\Q{{\mathbb{Q}}}
\def\Z{\mathbb{Z}}
\def\X{\mathbb{X}}
\def\Y{\mathbb{Y}}
\def\F{\mathcal{F}}
\def\L{\mathcal{L}}
\def\S{\mathcal{S}}
\def\supp{\operatorname{supp}}
\def\essi{\operatornamewithlimits{ess\,inf}}
\def\esss{\operatornamewithlimits{ess\,sup}}

\numberwithin{equation}{section}
\numberwithin{thm}{section}
\numberwithin{definition}{section}
\numberwithin{equation}{section}

\def\earrow{{\mathbf e}}
\def\rarrow{{\mathbf r}}
\def\uarrow{{\mathbf u}}
\def\varrow{{\mathbf V}}
\def\tpar{T_{\rm par}}
\def\apar{A_{\rm par}}

\def\reals{{\mathbb R}}
\def\torus{{\mathbb T}}
\def\t{{\mathcal T}}
\def\heis{{\mathbb H}}
\def\integers{{\mathbb Z}}
\def\z{{\mathbb Z}}
\def\naturals{{\mathbb N}}
\def\complex{{\mathbb C}\/}
\def\distance{\operatorname{distance}\,}
\def\support{\operatorname{support}\,}
\def\dist{\operatorname{dist}\,}
\def\Span{\operatorname{span}\,}
\def\degree{\operatorname{degree}\,}
\def\kernel{\operatorname{kernel}\,}
\def\dim{\operatorname{dim}\,}
\def\codim{\operatorname{codim}}
\def\trace{\operatorname{trace\,}}
\def\Span{\operatorname{span}\,}
\def\dimension{\operatorname{dimension}\,}
\def\codimension{\operatorname{codimension}\,}
\def\nullspace{\scriptk}
\def\kernel{\operatorname{Ker}}
\def\ZZ{ {\mathbb Z} }
\def\p{\partial}
\def\rp{{ ^{-1} }}
\def\Re{\operatorname{Re\,} }
\def\Im{\operatorname{Im\,} }
\def\ov{\overline}
\def\eps{\varepsilon}
\def\lt{L^2}
\def\diver{\operatorname{div}}
\def\curl{\operatorname{curl}}
\def\etta{\eta}
\newcommand{\norm}[1]{ \|  #1 \|}
\def\expect{\mathbb E}
\def\bull{$\bullet$\ }

\def\blue{\color{blue}}
\def\red{\color{red}}

\def\xone{x_1}
\def\xtwo{x_2}
\def\xq{x_2+x_1^2}
\newcommand{\abr}[1]{ \langle  #1 \rangle}

\newcommand{\Norm}[1]{ \left\|  #1 \right\| }
\newcommand{\set}[1]{ \left\{ #1 \right\} }
\newcommand{\ifou}{\raisebox{-1ex}{$\check{}$}}
\def\one{\mathbf 1}
\def\whole{\mathbf V}
\newcommand{\modulo}[2]{[#1]_{#2}}
\def \essinf{\mathop{\rm essinf}}
\def\scriptf{{\mathcal F}}
\def\scriptg{{\mathcal G}}
\def\m{{\mathcal M}}
\def\scriptb{{\mathcal B}}
\def\scriptc{{\mathcal C}}
\def\scriptt{{\mathcal T}}
\def\scripti{{\mathcal I}}
\def\scripte{{\mathcal E}}
\def\V{{\mathcal V}}
\def\scriptw{{\mathcal W}}
\def\scriptu{{\mathcal U}}
\def\scriptS{{\mathcal S}}
\def\scripta{{\mathcal A}}
\def\scriptr{{\mathcal R}}
\def\scripto{{\mathcal O}}
\def\scripth{{\mathcal H}}
\def\scriptd{{\mathcal D}}
\def\scriptl{{\mathcal L}}
\def\scriptn{{\mathcal N}}
\def\scriptp{{\mathcal P}}
\def\scriptk{{\mathcal K}}
\def\frakv{{\mathfrak V}}
\def\v{{\mathcal V}}
\def\C{\mathbb{C}}
\def\D{\mathcal{D}}
\def\R{\mathbb{R}}
\def\Rn{{\mathbb{R}^n}}
\def\rn{{\mathbb{R}^n}}
\def\Rm{{\mathbb{R}^{2n}}}
\def\r2n{{\mathbb{R}^{2n}}}
\def\Sn{{{S}^{n-1}}}
\def\bbM{\mathbb{M}}
\def\N{\mathbb{N}}
\def\Q{{\mathcal{Q}}}
\def\Z{\mathbb{Z}}
\def\F{\mathcal{F}}
\def\L{\mathcal{L}}
\def\G{\mathscr{G}}
\def\ch{\operatorname{ch}}
\def\supp{\operatorname{supp}}
\def\dist{\operatorname{dist}}
\def\essi{\operatornamewithlimits{ess\,inf}}
\def\esss{\operatornamewithlimits{ess\,sup}}
\def\dis{\displaystyle}
\def\dsum{\displaystyle\sum}
\def\dint{\displaystyle\int}
\def\dfrac{\displaystyle\frac}
\def\dsup{\displaystyle\sup}
\def\dlim{\displaystyle\lim}
\def\bom{\Omega}
\def\om{\omega}

\author[H. Yang]{Heng Yang}
\address{Heng Yang:
	College of Mathematics and System Sciences \\
	Xinjiang University \\
	Urumqi 830017 \\
	China}
\email{yanghengxju@yeah.net}

\author[J. Zhou]{Jiang Zhou$^{*}$}
\address{Jiang Zhou:
	College of Mathematics and System Sciences \\
	Xinjiang University \\
	Urumqi 830017 \\
	China}
\email{zhoujiang@xju.edu.cn}


\thanks{This work was supported by the National Natural Science Foundation of China (No.12061069).
\thanks{$^{*}$ Corresponding author, e-mail address: zhoujiang@xju.edu.cn}}

\date{November 27, 2023.}
\title[ Necessary and sufficient conditions for boundedness of commutators...  ]
{\bf Necessary and sufficient conditions for boundedness of commutators of fractional integral operators on slice spaces}

\begin{abstract}
Let $0<t<\infty$, $0<\alpha<n$, $1<p<r<\infty$ and $1<q<s<\infty$. In this paper, we prove that  $b\in B M O\left(\mathbb{R}^{n}\right)$  if and only if the commutator $[b, T_{\Omega,\alpha}]$  generated by the fractional integral operator with the rough kernel $T_{\Omega,\alpha}$ and  the locally integrable function $b$ is bounded from the slice space  $(E_{p}^{q})_{t}(\mathbb{R}^{n})$ to $(E_{r}^{s})_{t}(\mathbb{R}^{n})$. Meanwhile, we also show that $b\in \Lip_\beta(\mathbb{R}^{n}) $($0<\beta<1)$  if and only if the commutator  $\left[b, T_{\Omega,\alpha}\right]$ is bounded from   $(E_{p}^{q})_{t}(\mathbb{R}^{n})$ to $(E_{r}^{s})_{t}(\mathbb{R}^{n})$.
\end{abstract}

\maketitle

\section{Introduction and main results}\label{sec1}

Let $T$ be the classical singular integral operator and $b$ be the locally integrable function, the commutator $[b, T]$ is defined by
$$
[b,T]f(x)=b(x)T f(x)-T(bf)(x).
$$
A significant result of
Coifman, Rochberg and Weiss \cite{Cf1} stated that $b\in BMO(\rn)$ if and only if the commutator $[b, T]$  is bounded
on $L^p(\rn)$ for $1<p<\infty$; see also \cite{JS}.
In 1997, Ding \cite{DY} showed that $b\in BMO(\rn)$
if and only if the commutator $[b, T]$ is bounded on Morrey spaces. Later,
Tao, Yang, Yuan and Zhang \cite{TYYZ} demonstrated that $b\in BMO(\rn)$ if and only if the commutator generated by the Calder\'on--Zygmund operator with
the rough kernel and the locally integrable function $b$
is bounded on ball Banach function spaces by the extrapolation.
In 2022, Lu, Zhou and Wang\cite{LZW} obtained that $b\in BMO(\rn)$ if and only if the commutator generated by the Calder\'on--Zygmund operator with
the rough kernel and the locally integrable function $b$
is bounded on slice spaces using a direct method instead of extrapolation.
Recently, many authors have conducted extensive studies on the theory of commutators,
we refer the readers to see \cite{YZ3,ZZ,TYY1,HG} and therein references.

In 1926,
Wiener \cite{NW1} first introduced the concept of amalgam spaces as part of his work on generalized harmonic analysis. In 2019, Auscher and Mourgoglou \cite{AM} introduced the slice space $E_t^q(\rn)$ for $0<t<\infty$ and $1<q<\infty$ to explore weak solutions to
boundary value problems involving $t$-independent elliptic systems in the upper half plane.
The terminology, slice space, is inspired by the conceptual image of cutting through the tent space norm at a particular height.
This space is actually a special case within the broader category of Wiener amalgam spaces.

Auscher and Prisuelos-Arribas \cite{AP} recently investigated
the boundedness of some classical operators in harmonic analysis, including the Hardy--Littlewood maximal operator, the Calder\'on--Zygmund operator with the standard kernel, the Riesz potential and the Riesz transform associated with the second order divergence form elliptic operator, on the slice space $(E_p^q)_t(\rn)$ with $0<t<\infty$ and $1<p,q<\infty$ (see Definition \ref{de2.3}).
Ho \cite{KPH} demonstrated the boundedness of the Calder\'on--Zygmund operator with the standard kernel by employing extrapolation on Orlicz-slice spaces introduced in \cite{ZYW}.
Tao, Yang, Yuan and Zhang \cite{TYYZ} investigated the boundedness of the Calder\'on--Zygmund operator with the rough kernel and its commutator on ball Banach function spaces which include slice spaces. For more works on slice spaces, we can see \cite{AKM,AMP,YZ,YZ2} for example.

As usual, let $Q := Q(x_{0},r)$ denote the cube centered at $x_{0} \in \mathbb{R}^{n}$  with side length $r > 0$. We define $|B|$ as the Lebesgue measure of the cube $Q$ and let \( \chi_{Q} \) represent the characteristic function of the cube $Q$.  For $1 \leq p < \infty$, we define the conjugate index of $p$ as $p^{\prime} = \frac{p}{p-1}$. We will use the symbol $C$ to refer to a positive constant that is independent of the main parameters, but it may vary from line to line. If $f$ is a Lipschitz function, then $|f(x)-f(y)| \leq C|x-y|$. The notation $f \lesssim g$ indicates that  $f \leq C g$. If  $f \lesssim g$ and  $g \lesssim f$, we write $f \approx g$.

The space of functions with bounded mean oscillation $BMO(\rn)$ was introduced by John
and Nirenberg in \cite{JN}.

\begin{definition}
The bounded mean oscillation space, denoted \( BMO(\mathbb{R}^n) \), consists of all locally integrable functions \( f \) on \( \mathbb{R}^n \) that satisfy the following condition:
$$
\|f\|_{BMO(\rn)}:=\sup\limits_{Q}\frac1{|Q|}\int_Q|f(x)-f_Q|dx<\fz,
$$
where the supremum is taken over all cubes $Q$ in $\rn$ and $f_{Q}:=\frac{1}{|Q|} \int_{Q} f(x) d x$.
\end{definition}

\begin{definition}
For \( 0 < \beta < 1 \), a function \( b \) is said to belong to the Lipschitz space \( \mathrm{Lip}_\beta(\mathbb{R}^n) \), written \( b \in \mathrm{Lip}_\beta(\mathbb{R}^n) \), if there exists a constant \( C > 0 \) such that for every \( x, y \in \mathbb{R}^n \),
$$|b(x)-b(y)| \leq C|x-y|^{\beta}.$$
The infimum of such constants \( C \) is referred to as the \( \mathrm{Lip}_\beta(\mathbb{R}^n) \) norm of \( b \) and is denoted by \( \|b\|_{\mathrm{Lip}_\beta(\mathbb{R}^n)} \).
\end{definition}

For \( 0 < p < \infty \), the Lebesgue space \( L^p(\mathbb{R}^n) \) consists of all integrable functions \( f \) on \( \mathbb{R}^n \) for which
\[
\|f\|_{L^p(\mathbb{R}^n)} := \left( \int_{\mathbb{R}^n} |f(x)|^p \, \mathrm{d}x \right)^{\frac{1}{p}} < \infty.
\]

\begin{definition} \label{de2.3}
For \( 0 < t < \infty \) and \( 1 < p, q < \infty \), the slice space \( (E_{p}^{q})_{t}(\mathbb{R}^n) \) is defined as the collection of all locally \( p \)-integrable functions \( f \) on \( \mathbb{R}^n \) such that
$$
\|f\|_{(E_{p}^{q})_{t}(\mathbb{R}^{n})}=\left(\int_{\mathbb{R}^n}\left(\frac{1}{|Q(x, t)|}\int_{Q(x, t)}|f(y)|^{p} \mathrm{d} y\right)^{\frac{q}{p}} \mathrm{d} x\right)^{\frac{1}{q}}<\infty.
$$
\end{definition}

When \( p = q \), the slice space \( (E_{p}^{q})_{t}(\mathbb{R}^n) \) coincides with the Lebesgue space \( L^p(\mathbb{R}^n) \).

We recall that \( T_{\Omega, \alpha} \) denotes the fractional integral operator with the rough kernel (see, for example, \cite{LDY}) and is defined by
$$
T_{\Omega,\alpha}(f)(x):=\int_{\mathbb R^n} \frac{\Omega(x-y)}{|x-y|^{n-\alpha}}f(y) dy
$$
for any locally integrable function $f$ and any $x\in\mathbb R^n$, where the function $\Omega$ satisfies the condition

\begin{equation}\label{Eq2}
\Omega(\lambda x)=\Omega(x) ~~\text{for any }\lambda\in(0,\infty) ~\text{and}~ x\in \mathbb S^{n-1}\tag{1.1}
\end{equation}
and
\begin{equation}\label{Eq3}
\int_{\mathbb S^{n-1}}\Omega(x)d\sigma(x)=0,\tag{1.2}
\end{equation}
whenever $\mathbb S^{n-1}:=\{x\in\mathbb R^n: |x|=1\}$ represents the unit sphere in $\mathbb R^n$ and $d\sigma$ denotes the area measure on $\mathbb S^{n-1}$.

We then have the following conclusions.
\begin{theorem}\label{TH1.1}
Let $0<\alpha<n$, $0<t<\infty$, $1<p<r<\infty$ and $1<q<s<\infty$ with $\frac{\alpha}{n}=\frac{1}{p}-\frac{1}{r}=\frac{1}{q}-\frac{1}{s}$. Let \( T_{\Omega, \alpha} \) denote the fractional integral operator with the rough kernel \( \Omega \), which satisfies (\ref{Eq2}) and (\ref{Eq3}), and let \( b \) be a locally integrable function. Then the following statements hold true.

(a) Suppose that $\Omega$ is a Lipschitz function. If $b\in BMO(\mathbb{R}^{n})$, then $[b,T_{\Omega,\alpha}]$ is bounded from $(E_p^q)_t(\mathbb{R}^{n})$ to $(E_r^s)_t(\mathbb{R}^{n})$;

(b) In addition, suppose  that $\Omega$ is infinitely differentiable on an open set. If $[b,T_{\Omega,\alpha}]$ is bounded from $(E_p^q)_t(\mathbb{R}^{n})$ to $(E_r^s)_t(\mathbb{R}^{n})$, then $b\in BMO(\mathbb{R}^{n})$.
\end{theorem}

\begin{theorem}\label{TH1.2}
Let $0<\beta<1$, $0<\alpha<n$, $\alpha+\beta<n$, $0<t<\infty$, $1<p<r<\infty$ and $1<q<s<\infty$ with $\frac{\alpha+\beta}{n}=\frac{1}{p}-\frac{1}{r}=\frac{1}{q}-\frac{1}{s}$.
Let \( T_{\Omega, \alpha} \) denote the fractional integral operator with the rough kernel \( \Omega \), which satisfies (\ref{Eq2}) and (\ref{Eq3}), and let \( b \) be a locally integrable function. Then the following statements hold true.

(a) Suppose that $\Omega$ is a Lipschitz function. If $b\in \Lip_\beta(\mathbb{R}^{n})$, then $[b,T_{\Omega,\alpha}]$ is bounded from $(E_p^q)_t(\mathbb{R}^{n})$ to $(E_r^s)_t(\mathbb{R}^{n})$;

(b) In addition, suppose  that $\Omega$ is infinitely differentiable on an open set. If $[b,T_{\Omega,\alpha}]$ is bounded  from $(E_p^q)_t(\mathbb{R}^{n})$ to $(E_r^s)_t(\mathbb{R}^{n})$, then $b\in \Lip_\beta(\mathbb{R}^{n}) $.
\end{theorem}

\section{Preliminaries}

In this section, we provide essential lemmas and definitions that are crucial for establishing our main results.

\begin{lemma}\emph{\cite{DS}}  \label{le2.5}
For \( 0 < \beta < 1 \) and \( 1 \leq q < \infty \), the space \( \mathrm{Lip}_{\beta, q}(\mathbb{R}^n) \) is defined as the collection of all locally integrable functions \( f \) that satisfy
$$
\begin{aligned}
&\|f\|_{\Lip_{\beta, q}\left(\mathbb{R}^{n}\right)}=\sup _{Q} \frac{1}{|Q|^{\frac{\beta}{n}}}\left(\frac{1}{|Q|} \int_{Q}\left|f(x)-f_{Q}\right|^{q}dx\right)^{\frac{1}{q}}<\infty.
\end{aligned}
$$
Then, for all $0<\beta<1$ and $1 \leq q<\infty$, the spaces $\Lip_{\beta}\left(\mathbb{R}^{n}\right)$ and $\Lip_{\beta, q}\left(\mathbb{R}^{n}\right)$ coincide and have equivalent norms.
\end{lemma}

\begin{lemma} \emph{\cite{LDY}}\label{le2.1}
Suppose that $0<\alpha<n$, $1<p<\frac{n}{\alpha}$ and $\frac{1}{q}=\frac{1}{p}-\frac{\alpha}{n}$.
Let \( T_{\Omega, \alpha} \) denote the fractional integral operator with the rough kernel \( \Omega \), which satisfies (\ref{Eq2}) and (\ref{Eq3}). Additionally, \( \Omega \) is a Lipschitz function.
If $f \in L^p\left(\mathbb{R}^n\right)$, then $$\|T_{\Omega, \alpha} f\|_{L^q\left(\mathbb{R}^n\right)} \leq C\|f\|_{L^p\left(\mathbb{R}^n\right)}.$$
\end{lemma}

\begin{lemma} \emph{\cite{LZW}}\label{le2.3}
Let $0<t<\infty$, $1<p,q<\infty$. Then there exists a constant $C$ such that
$$
\lf\|Mf\right\|_{(E_p^q)_t(\mathbb R^n)}\leq C\|M^\sharp f\|_{(E_p^q)_t(\mathbb R^n)},
$$
where the constant $C$ does not depend on $f$ and $t$.
\end{lemma}

\begin{lemma} \label{le2.4}\emph{\cite{LWZ}}
Let $0\leq\alpha<n$, $0<t<\infty$, $1<p<r<\fz$ and $1<q<s<\fz$ with $\frac{\alpha}{n}=\frac{1}{p}-\frac{1}{r}=\frac{1}{q}-\frac{1}{s}$ . If $f\in(E_p^q)_t(\rn)$, then
$$
\|M_\az f\|_{(E_r^s)_t(\rn)}\leq C \|f\|_{(E_p^q)_t(\rn)},
$$
where the constant $C$ does not depend on $f$ and $t$.
\end{lemma}

\begin{lemma}\emph{\cite{LZW}} \label{le2.2}
Let $0<t<\infty$, $1<p,q<\infty$ and $Q$ be a cube in $\mathbb{R}^{n}$. Then
$$
\|\chi_{Q}\|_{(E_p^q)_t(\mathbb{R}^n)} \approx |Q|^{\frac{1}{q}}.
$$
\end{lemma}

\begin{lemma}  \label{le2.6}
Let $0<\alpha<n, 1<p<\frac{n}{\alpha}$ and $\frac{1}{q}=\frac{1}{p}-\frac{\alpha}{n}$.
Let \( T_{\Omega, \alpha} \) denote the fractional integral operator with the rough kernel \( \Omega \), which satisfies (\ref{Eq2}) and (\ref{Eq3}). Additionally, suppose that \( \Omega \) is a Lipschitz function. Then for $1<\eta<\frac{n}{\alpha}$ and $b \in \mathrm{BMO}(\mathbb{R}^{n})$, there exists a constant $C$ that does not depend on $f$ and $b$, such that
$$
\begin{gathered}
M^{\sharp}\left(\left[b, T_{\Omega,\alpha}\right] f\right)(x) \leq C\|b\|_{\mathrm{BMO}(\mathbb{R}^{n})}\left[\left(M\left(\left|T_{\Omega,\alpha} f\right|^{\eta}\right)(x)\right)^{\frac{1}{\eta}}+\left(M_{\alpha \eta}\left(|f|^{\eta}\right)(x)\right)^{\frac{1}{\eta}}\right].
\end{gathered}
$$
\end{lemma}

\begin{proof}
Let \( Q \) be a fixed cube, and define
$$
\begin{aligned}
\left[b, T_{\Omega,\alpha}\right] f(x) & =\left(b(x)-b_{2Q}\right) T_{\Omega,\alpha} f(x)-T_{\Omega,\alpha}\left(\left(b-b_{2Q}\right) f \chi_{2 Q}\right)(x)-T_{\Omega,\alpha}\left(\left(b-b_{2Q}\right) f \chi_{(2 Q)^{c}}\right)(x) \\
&:=I_{1}(x)-I_{2}(x)-I_{3}(x).
\end{aligned}
$$
By applying H\"{o}lder's inequality, we obtain
$$
\begin{aligned}
\frac{1}{|Q|} \int_{Q}\left|I_{1}(x)\right| d x & \leq\left(\frac{1}{|Q|} \int_{Q}\left|b(x)-b_{2Q}\right|^{\eta^{\prime}} d x\right)^{\frac{1}{\eta^{\prime}}}\left(\frac{1}{|Q|} \int_{Q}\left|T_{\Omega,\alpha} f(x)\right|^{\eta} d x\right)^{\frac{1}{\eta}} \\
& \leq C\left(\frac{1}{|2Q|} \int_{2Q}\left|b(x)-b_{2Q}\right|^{\eta^{\prime}} d x\right)^{\frac{1}{\eta^{\prime}}}\left(\frac{1}{|2Q|} \int_{2Q}\left|T_{\Omega,\alpha} f(x)\right|^{\eta} d x\right)^{\frac{1}{\eta}} \\
& \leq C\|b\|_{\mathrm{BMO}(\mathbb{R}^{n}) }\left(M\left(|T_{\Omega,\alpha} f|^{\eta}\right)\right)^{\frac{1}{\eta}}.
\end{aligned}
$$
Given that \( 1 < \eta < p \), we can select \( \gamma > 1 \) and \( \delta > 1 \) such that \( \gamma \delta = \eta \). This implies \( 1 < \delta < p < \frac{n}{\alpha} \). Furthermore, there exists \( u > \delta \) such that \( \frac{1}{u} = \frac{1}{\delta} - \frac{\alpha}{n} \). By Lemma \ref{le2.1}, we have
$$
\begin{aligned}
\frac{1}{|Q|} \int_{Q}\left|I_{2}(x)\right| d x & \leq\left(\frac{1}{|Q|} \int_{Q}\left|T_{\Omega,\alpha}\left(\left(b-b_{Q}\right) f \chi_{2 Q}\right)(x)\right|^{u} d x\right)^{\frac{1}{u}} \\
& \leq C \frac{1}{|Q|^{\frac{1}{u}}}\left(\int_{2 Q}\left|b(x)-b_{Q}\right|^{\delta}|f(x)|^{\delta} d x\right)^{\frac{1}{\delta}} \\
& \leq C \frac{1}{|Q|^{\frac{1}{u}}}\left(\int_{2 Q}\left|b(x)-b_{Q}\right|^{\gamma^{\prime} \delta} d x\right)^{\frac{1}{\delta \gamma^{\prime}}}\left(\int_{2 Q}|f(x)|^{\gamma \delta} d x\right)^{\frac{1}{\gamma \delta}} \\
& \leq C\|b\|_{\mathrm{BMO}(\mathbb{R}^{n}) }|Q|^{\frac{\alpha}{n}}\left(\frac{1}{|Q|} \int_{2 Q}|f(x)|^{\eta} d x\right)^{\frac{1}{\eta}} \\
& =C\|b\|_{\mathrm{BMO}(\mathbb{R}^{n}) }\left(\frac{1}{|Q|^{1-\frac{\alpha \eta}{n}}} \int_{2 Q}|f(x)|^{\eta} d x\right)^{\frac{1}{\eta}} \\
& \leq C\|b\|_{\mathrm{BMO}(\mathbb{R}^{n}) }\left(M_{\alpha \eta}\left(|f|^{\eta}\right)(x)\right)^{\frac{1}{\eta}} .
\end{aligned}
$$
For \( x \in Q \) and \( y \in (2Q)^c \), it follows that \( |x - y| \approx |x_0 - y| \). Since \( \Omega \) is a Lipschitz function, we can then conclude that
$$
\begin{aligned}
&\left|T_{\Omega,\alpha}\left(\left(b-b_{2Q}\right) f \chi_{(2 Q)^{c}}\right)(x)-T_{\Omega,\alpha}\left(\left(b-b_{2Q}\right) f \chi_{(2 Q)^{c}}\right)\left(x_{0}\right)\right| \\
&\leq  \int_{\mathbb{R}^{n} \backslash 2 Q}\left|\frac{\Omega(x-y)}{|x-y|^{n-\alpha}}-\frac{\Omega(x_{0}-y)}{\left|x_{0}-y\right|^{n-\alpha}}\right|\left|b(y)-b_{2Q}\right||f(y)| d y \\
&\leq  \int_{\mathbb{R}^{n} \backslash 2 Q}\left|\Omega(y-z)\right|\left|\frac{1}{|y-z|^{n-\alpha}}-\frac{1}{\left|x_{0}-z\right|^{n-\alpha} \mid}\right|\left|b(y)-b_{2Q}\right||f(y)| d y \\
&\quad+  \int_{\mathbb{R}^{n} \backslash 2 Q}\frac{1}{\left|x_{0}-z\right|^{n-\alpha}}\left|\Omega\left(\frac{y-z}{|y-z|}\right)-\Omega\left(\frac{x_{0}-z}{\left|x_{0}-z\right|}
\right)\right|\left|b(y)-b_{2Q}\right||f(y)| d y \\
\end{aligned}
$$
$$
\begin{aligned}
&\leq  C\int_{\mathbb{R}^{n} \backslash 2 Q}\|\Omega\|_{L^{\infty}\left(\mathbb{S}^{n-1}\right)} \frac{\left|x-x_{0}\right|}{\left|x_{0}-z\right|^{n-\alpha+1}}\left|b(y)-b_{2Q}\right||f(y)| d y \\
&\quad+  C\int_{\mathbb{R}^{n} \backslash 2 Q}\frac{\left|x-x_{0}\right|}{\left|x_{0}-z\right|^{n-\alpha+1}}\left|b(y)-b_{2Q}\right||f(y)| d y \\
&\leq  C \int_{\mathbb{R}^{n} \backslash 2 Q} \frac{\left|x-x_{0}\right|}{\left|x_{0}-y\right|^{n-\alpha+1}}\left|b(y)-b_{2Q}\right||f(y)| d y \\
&\leq  C\left(\int_{\mathbb{R}^{n} \backslash 2 Q} \frac{\left|x-x_{0}\right|}{\left|x_{0}-y\right|^{n+1}}\left|b(y)-b_{2Q}\right|^{\eta^{\prime}} d y\right)^{\frac{1}{\eta^{\prime}}} \left(\int_{\mathbb{R}^{n} \backslash 2 Q} \frac{\left|x-x_{0}\right|}{\left|x_{0}-y\right|^{n+1-\alpha \eta}}|f(y)|^{\eta} d y\right)^{\frac{1}{\eta}} \\
&\leq  C\sum_{j=1}^{\infty} \frac{r}{{(2^{j}r)}} \left(\fint_{2^{j+1} Q}|b(y)-b_{2Q}|^{\eta^{\prime}} d y\right)^{\frac{1}{\eta^{\prime}}} \left(\frac{1}{|2^{j+1} Q|^{1-\frac{\alpha \eta}{n}}}\int_{2^{j+1} Q}|f(y)|^{\eta} d y\right)^{\frac{1}{\eta}}\\
&\leq  C\|b\|_{\mathrm{BMO}(\mathbb{R}^{n}) }\sum_{j=1}^{\infty} \frac{j}{{2^{j}}} \left(\frac{1}{|2^{j+1}Q|^{1-\frac{\alpha \eta}{n}}}\int_{2^{j+1} Q}|f(y)|^{\eta} d y\right)^{\frac{1}{\eta}}\\
&\leq  C\|b\|_{\mathrm{BMO}(\mathbb{R}^{n}) }\left(M_{\alpha \eta}\left(|f|^{\eta}\right)(x)\right)^{\frac{1}{\eta}} .
\end{aligned}
$$
From the previous estimates, it follows that
$$
\begin{aligned}
& \frac{1}{|Q|} \int_{Q}\left|\left[b, T_{\Omega,\alpha}\right](f)(x)-T_{\Omega,\alpha}\left(\left(b-b_{2Q}\right) f \chi_{(2 Q)^{c}}\right)\left(x_{0}\right)\right| d x \\
& \leq \frac{1}{|Q|} \int_{Q}\left|I_{1}(x)\right| d x+\frac{1}{|Q|} \int_{Q}\left|I_{2}(x)\right| d x +\frac{1}{|Q|} \int_{Q}\left|I_{3}(x)-T_{\Omega,\alpha}\left(\left(b-b_{Q}\right) f \chi_{(2 Q)^{c}}\right)\left(x_{0}\right)\right| d x \\
& \leq C\|b\|_{\mathrm{BMO}(\mathbb{R}^{n})}\left[\left(M\left(\left|T_{\Omega,\alpha} f\right|^{\eta}\right)(x)\right)^{\frac{1}{\eta}}+\left(M_{\alpha \eta}\left(|f|^{\eta}\right)(x)\right)^{\frac{1}{\eta}}\right].
\end{aligned}
$$
This concludes the proof of Lemma \ref{le2.6}.
\end{proof}

\begin{lemma} \label{le2.7}
Let $0<\alpha<n$, $0<\beta<1$, $\alpha+\beta<n$, $1<p<\frac{n}{\alpha+\beta}$ and $\frac{1}{q}=\frac{1}{p}-\frac{\alpha+\beta}{n}$.
Let \( T_{\Omega, \alpha} \) denote the fractional integral operator with the rough kernel \( \Omega \), which satisfies (\ref{Eq2}) and (\ref{Eq3}). Additionally, suppose that \( \Omega \) is a Lipschitz function. Then for $1<\eta<\frac{n}{\alpha+\beta}$ and $b\in \Lip_\beta(\rn)$, there exists a constant $C$ that does not depend on $f$ and $b$, such that
$$
\begin{gathered}
M^{\sharp}\left(\left[b, T_{\Omega,\alpha}\right] f\right)(x) \leq C\|b\|_{\Lip_\beta(\mathbb{R}^{n})}\left[\left(M_{\beta \eta}\left(|T_{\Omega,\alpha} f|^{\eta}\right)\right)^{\frac{1}{\eta}}+\left(M_{\alpha\eta+\beta \eta}\left(|f|^{\eta}\right)(x)\right)^{\frac{1}{\eta}}\right].
\end{gathered}
$$
\end{lemma}

\begin{proof}
Let \( Q \) be a fixed cube and define
$$
\begin{aligned}
\left[b, T_{\Omega,\alpha}\right] f(x) & =\left(b(x)-b_{2Q}\right) T_{\Omega,\alpha} f(x)-T_{\Omega,\alpha}\left(\left(b-b_{2Q}\right) f \chi_{2 Q}\right)(x)-T_{\Omega,\alpha}\left(\left(b-b_{2Q}\right) f \chi_{(2 Q)^{c}}\right)(x) \\
&:=II_{1}(x)-II_{2}(x)-II_{3}(x).
\end{aligned}
$$
Using H\"{o}lder's inequality, it follows that
$$
\begin{aligned}
\frac{1}{|Q|} \int_{Q}\left|II_{1}(x)\right| d x & \leq\left(\frac{1}{|Q|} \int_{Q}\left|b(x)-b_{2Q}\right|^{\eta^{\prime}} d x\right)^{\frac{1}{\eta^{\prime}}}\left(\frac{1}{|Q|} \int_{Q}\left|T_{\Omega,\alpha} f(x)\right|^{\eta} d x\right)^{\frac{1}{\eta}} \\
& \leq C\left(\frac{1}{|2Q|} \int_{2Q}\left|b(x)-b_{2Q}\right|^{\eta^{\prime}} d x\right)^{\frac{1}{\eta^{\prime}}}\left(\frac{1}{|2Q|} \int_{2Q}\left|T_{\Omega,\alpha} f(x)\right|^{\eta} d x\right)^{\frac{1}{\eta}} \\
& \leq C\frac{1}{|2Q|^{\frac{\beta}{n}}}\left(\frac{1}{|2Q|} \int_{2Q}\left|b(x)-b_{2Q}\right|^{\eta^{\prime}} d x\right)^{\frac{1}{\eta^{\prime}}}\left(\frac{1}{|2Q|^{1-\frac{\beta \eta}{n}}} \int_{2Q}\left|T_{\Omega,\alpha} f(x)\right|^{\eta} d x\right)^{\frac{1}{\eta}} \\
& \leq C\|b\|_{\Lip_\beta(\mathbb{R}^{n})}\left(M_{\beta \eta}\left(|T_{\Omega,\alpha} f|^{\eta}\right)\right)^{\frac{1}{\eta}}.
\end{aligned}
$$
Since \( 1 < \eta < p \), we can select \( \gamma > 1 \) and \( \delta > 1 \) such that \( \gamma \delta = \eta \). This implies \( 1 < \delta < p < \frac{n}{\alpha} \). Moreover, there exists \( u > \delta \) such that \( \frac{1}{u} = \frac{1}{\delta} - \frac{\alpha}{n} \). Applying Lemma \ref{le2.1}, we obtain that
$$
\begin{aligned}
\frac{1}{|Q|} \int_{Q}\left|II_{2}(x)\right| d x & \leq\left(\frac{1}{|Q|} \int_{Q}\left|T_{\Omega,\alpha}\left(\left(b-b_{Q}\right) f \chi_{2 Q}\right)(x)\right|^{u} d x\right)^{\frac{1}{u}} \\
& \leq C \frac{1}{|Q|^{\frac{1}{u}}}\left(\int_{2 Q}\left|b(x)-b_{Q}\right|^{\delta}|f(x)|^{\delta} d x\right)^{\frac{1}{\delta}} \\
& \leq C \frac{1}{|Q|^{\frac{1}{u}}}\left(\int_{2 Q}\left|b(x)-b_{Q}\right|^{\gamma^{\prime} \delta} d x\right)^{\frac{1}{\delta \gamma^{\prime}}}\left(\int_{2 Q}|f(x)|^{\gamma \delta} d x\right)^{\frac{1}{\gamma \delta}} \\
& \leq C\|b\|_{\Lip_\beta(\mathbb{R}^{n})  }|Q|^{\frac{\alpha}{n}}\left(\frac{1}{|Q|^{1-\frac{\beta \eta}{n}}} \int_{2 Q}|f(x)|^{\eta} d x\right)^{\frac{1}{\eta}} \\
& =C\|b\|_{\Lip_\beta(\mathbb{R}^{n}) }\left(\frac{1}{|Q|^{1-\frac{\alpha\eta+\beta \eta}{n}}} \int_{2 Q}|f(x)|^{\eta} d x\right)^{\frac{1}{\eta}} \\
& \leq C\|b\|_{\Lip_\beta(\mathbb{R}^{n}) }\left(M_{\alpha\eta+\beta \eta}\left(|f|^{\eta}\right)(x)\right)^{\frac{1}{\eta}} .
\end{aligned}
$$
For \( x \in Q \) and \( y \in (2Q)^c \), it follows that \( |x - y| \approx |x_0 - y| \). Thus, we observe that
$$
\begin{aligned}
&\left|T_{\Omega,\alpha}\left(\left(b-b_{2Q}\right) f \chi_{(2 Q)^{c}}\right)(x)-T_{\Omega,\alpha}\left(\left(b-b_{2Q}\right) f \chi_{(2 Q)^{c}}\right)\left(x_{0}\right)\right| \\
&\leq  \int_{\mathbb{R}^{n} \backslash 2 Q}\left|\frac{\Omega(x-y)}{|x-y|^{n-\alpha}}-\frac{\Omega(x_{0}-y)}{\left|x_{0}-y\right|^{n-\alpha}}\right|\left|b(y)-b_{2Q}\right||f(y)| d y \\
&\leq  \int_{\mathbb{R}^{n} \backslash 2 Q}\left|\Omega(y-z)\right|\left|\frac{1}{|y-z|^{n-\alpha}}-\frac{1}{\left|x_{0}-z\right|^{n-\alpha} \mid}\right|\left|b(y)-b_{2Q}\right||f(y)| d y \\
&\quad+  \int_{\mathbb{R}^{n} \backslash 2 Q}\frac{1}{\left|x_{0}-z\right|^{n-\alpha}}\left|\Omega\left(\frac{y-z}{|y-z|}\right)-\Omega\left(\frac{x_{0}-z}{\left|x_{0}-z\right|}
\right)\right|\left|b(y)-b_{2Q}\right||f(y)| d y \\
&\leq  C\int_{\mathbb{R}^{n} \backslash 2 Q}\|\Omega\|_{L^{\infty}\left(\mathbb{S}^{n-1}\right)} \frac{\left|x-x_{0}\right|}{\left|x_{0}-z\right|^{n-\alpha+1}}\left|b(y)-b_{2Q}\right||f(y)| d y \\
&\quad+  C\int_{\mathbb{R}^{n} \backslash 2 Q}\frac{\left|x-x_{0}\right|}{\left|x_{0}-z\right|^{n-\alpha+1}}\left|b(y)-b_{2Q}\right||f(y)| d y \\
&\leq  C \int_{\mathbb{R}^{n} \backslash 2 Q} \frac{\left|x-x_{0}\right|}{\left|x_{0}-y\right|^{n-\alpha+1}}\left|b(y)-b_{2Q}\right||f(y)| d y \\
&\leq  C\left(\int_{\mathbb{R}^{n} \backslash 2 Q} \frac{\left|x-x_{0}\right|}{\left|x_{0}-y\right|^{n+1}}\left|b(y)-b_{2Q}\right|^{\eta^{\prime}} d y\right)^{\frac{1}{\eta^{\prime}}} \left(\int_{\mathbb{R}^{n} \backslash 2 Q} \frac{\left|x-x_{0}\right|}{\left|x_{0}-y\right|^{n+1-\alpha \eta}}|f(y)|^{\eta} d y\right)^{\frac{1}{\eta}} \\
&\leq  C\sum_{j=1}^{\infty} \frac{r}{{(2^{j}r)}} \left(\fint_{2^{j+1} Q}|b(y)-b_{2Q}|^{\eta^{\prime}} d y\right)^{\frac{1}{\eta^{\prime}}} \left(\frac{1}{|2^{j+1} Q|^{1-\frac{\alpha\eta+\beta \eta}{n}}}\int_{2^{j+1} Q}|f(y)|^{\eta} d y\right)^{\frac{1}{\eta}}\\
&\leq  C\|b\|_{\Lip_\beta(\mathbb{R}^{n}) }\sum_{j=1}^{\infty} \frac{j}{{2^{j}}} \left(\frac{1}{|2^{j+1}Q|^{1-\frac{\alpha\eta+\beta \eta}{n}}}\int_{2^{j+1} Q}|f(y)|^{\eta} d y\right)^{\frac{1}{\eta}}\\
&\leq  C\|b\|_{\Lip_\beta(\mathbb{R}^{n})  }\left(M_{\alpha\eta+\beta \eta}\left(|f|^{\eta}\right)(x)\right)^{\frac{1}{\eta}} .
\end{aligned}
$$
Combining the above estimates, we have
$$
\begin{aligned}
& \frac{1}{|Q|} \int_{Q}\left|\left[b, T_{\Omega,\alpha}\right](f)(x)-T_{\Omega,\alpha}\left(\left(b-b_{2Q}\right) f \chi_{(2 Q)^{c}}\right)\left(x_{0}\right)\right| d x \\
& \leq \frac{1}{|Q|} \int_{Q}\left|II_{1}(x)\right| d x+\frac{1}{|Q|} \int_{Q}\left|II_{2}(x)\right| d x +\frac{1}{|Q|} \int_{Q}\left|II_{3}(x)-T_{\Omega,\alpha}\left(\left(b-b_{Q}\right) f \chi_{(2 Q)^{c}}\right)\left(x_{0}\right)\right| d x \\
& \leq C\|b\|_{\Lip_\beta(\mathbb{R}^{n})}\left[\left(M_{\beta \eta}\left(|T_{\Omega,\alpha} f|^{\eta}\right)\right)^{\frac{1}{\eta}}+\left(M_{\alpha\eta+\beta \eta}\left(|f|^{\eta}\right)(x)\right)^{\frac{1}{\eta}}\right] .
\end{aligned}
$$
Thus, we complete the proof of  Lemma \ref{le2.7}.
\end{proof}

\section{Proofs of Theorems \ref{TH1.1}-\ref{TH1.2}}

\begin{proof}[Proof of Theorem \ref{TH1.1}]
We begin by proving statement $(a)$. Let \( T_{\Omega, \alpha} \) be the fractional integral operator with the rough kernel \( \Omega \), which satisfies (\ref{Eq2}) and (\ref{Eq3}), and let \( \Omega \) be a Lipschitz function. Applying Lemma \ref{le2.3}, we obtain
$$
\lf\|[b,T_{\Omega,\alpha}]f\right\|_{(E_r^{s})_t(\rn)}\le\lf\|M([b,T_{\Omega,\alpha}]f)\right\|_{(E_r^{s})_t(\rn)}
\leq C\lf\|M^\sharp(|[b,T_{\Omega,\alpha}]f|)\right\|_{(E_r^{s})_t(\rn)}.
$$
For $1<\eta<\frac{n}{\alpha}$, Lemma \ref{le2.1}, Lemma \ref{le2.4} and Lemma \ref{le2.6} yield
\begin{align*}
\lf\|M^\sharp\lf(|[b,T_{\Omega,\alpha}]f|\right)\right\|_{(E_r^{s})_t(\rn)}
&\leq C\|b\|_{BMO(\rn)}\lf[\left\|\left(M\left(\left|T_{\Omega,\alpha} f\right|^{\eta}\right)(x)\right)^{\frac{1}{\eta}}\right\|_{(E_r^{s})_t(\rn)}
+\lf\|\left(M_{\alpha s}\left(|f|^{\eta}\right)(x)\right)^{\frac{1}{\eta}}\right\|_{(E_r^{s})_t(\rn)}\right]\\
&\leq C\|b\|_{BMO(\rn)}\|f\|_{(E_p^{q})_t(\rn)}.
\end{align*}
Therefore,
$$
\lf\|[b,T_{\Omega,\alpha}]f\right\|_{(E_r^{s})_t(\rn)}\leq C\lf\|b\right\|_{BMO(\rn)}\lf\|f\right\|_{(E_p^{q})_t(\rn)}.
$$

Next, we prove $(b)$. Assume that the commutator \( [b, T_{\Omega, \alpha}] \) is bounded from \( (E_r^s)_t(\mathbb{R}^n) \) to \( (E_p^q)_t(\mathbb{R}^n) \). We apply the same method as Janson \cite{JS}. Let \( K(x) := \frac{\Omega(x/|x|)}{|x|^{n-\alpha}} \). Choose a non-zero vector \( z_0 \in \mathbb{R}^n \) such that \( \frac{1}{K(x)} \) can be expressed as an absolutely convergent Fourier series in the neighborhood \( |x - z_0| \leq 2\sqrt{n} \),
$$
\frac 1{K(x)}\chi_{Q(z_0,2)}(x)=\sum_{m\in \mathbb{Z}^n}a_me^{2im\cdot x}\chi_{Q(z_0,2)}(x),
$$
with $\sum_{m\in \mathbb{Z}^n}|a_m|<\infty$.

Let \( x' = \delta^{-1} z_0 \), where \( \delta > 0 \). If \( |x - x'| < 2\sqrt{n} \), then we observe that
$$
\frac1{K(x)}=\frac{\dz^{-n}}{K(x\dz)}=\dz^{-n}\sum_{m\in \mathbb{Z}^n}a_me^{2im\cdot \dz x}\chi_{Q(z_0,2)}(x).
$$
Choose any cube $Q=Q(x_0,r)$ and $Q_{z_0}=Q(x_0+z_0r,r)$. Let $s(x)=\overline{sgn\lf(\int_{Q_{z_0}}b(x)-b(y)dy\right)}$.
If $x\in Q$ and $y\in Q_{z_0}$, then $\frac{y-x}{r}\in Q(z_0,2)$. This gives us,
\begin{align*}
\int_Q|b(x)-b_{Q_{z_0}}|dx&=\int_Q(b(x)-b_{Q_{z_0}})s(x)dx\\
&=|Q_{z_0}|^{-1}\int_Q\int_{Q_{z_0}}s(x)(b(x)-b(y))dydx\\
&=r^{-n}\int_Q\int_{Q_{z_0}}(b(x)-b(y))\frac{r^{n-\alpha}K(x-y)}{K(\frac{x-y}r)}s(x)\chi_Q(x)\chi_{Q_{z_0}}(y)dydx\\
&=r^{-\alpha}\sum_{m\in\mathbb{Z}^n}a_m\int_Q\int_{Q_{z_0}}(b(x)-b(y))K(x-y)f_m(y)g_m(x)dydx\\
&\leq Cr^{-\alpha}\sum_{m\in\mathbb{Z}^n}|a_m|\int_Q\lf|[b,T_{\Omega,\alpha}]f_m(x)g_m(x)\right|dx,
\end{align*}
where
$$
f_m(y)=e^{-2im\delta\cdot \frac{y}{r}}\chi_{Q_{z_0}}(y)
$$
and
$$
g_m(x)=e^{2im\delta\cdot \frac{x}{r}}s(x)\chi_Q(x).
$$
Applying H\"older's inequality,
$$
\int_Q|b(x)-b_{Q_{z_0}}|dx\leq Cr^{-\alpha}\sum_{m\in\mathbb{Z}^n}|a_m|\lf\|[b,T_{\Omega,\alpha}]f_m\right\|_{(E_r^s)_t(\rn)}
\lf\|g_m\right\|_{(E_{r'}^{s'})_t(\rn)},
$$
where $\frac{1}{r}+\frac{1}{r'}=1$ for $r>1$ and $\frac{1}{s}+\frac{1}{s'}=1$  for $s>1$.
Using Lemma \ref{le2.2} and the condition $\frac{\alpha}{n}=\frac{1}{q}-\frac{1}{s}$, we deduce that
\begin{align*}
\frac{1}{|Q|}\int_Q|b(x)-b_Q |dx&\le\frac{2}{|Q|}\int_Q|b(x)-b_{Q_{z_0}} |dx\\
&\leq Cr^{-\alpha}\frac{1}{|Q|}\sum_{m\in\mathbb{Z}^n}|a_m|\lf\|f_m\right\|_{(E_p^q)_t(\rn)}
\lf\|g_m\right\|_{(E_{r'}^{s'})_t(\rn)}\\
&\leq Cr^{-\alpha}\frac{1}{|Q|}\sum_{m\in\mathbb{Z}^n}|a_m|\|\chi_{Q_{z_0}}\|_{(E_p^q)_t(\rn)}
\|\chi_{Q}\|_{(E_{r'}^{s'})_t(\rn)}\\
&\leq C|Q|^{\frac{1}{q}-\frac{1}{s}-\frac{\alpha}{n}}\\
&\leq C.
\end{align*}
Then $b\in BMO(\rn)$.
This completes the proof of Theorem \ref{TH1.1}.
\end{proof}

\begin{proof}[Proof of Theorem \ref{TH1.2}]
We now prove that $(a)$. Using Lemma \ref{le2.1}, Lemma \ref{le2.3}, Lemma \ref{le2.4} and Lemma \ref{le2.7}, we have
\begin{align*}
\lf\|[b,T_\boz]f\right\|_{(E_r^{s})_t(\rn)}&\leq \lf\|M ([b,T_\boz]f)\right\|_{(E_r^{s})_t(\rn)}
\leq C\lf\|M^\sharp (|[b,T_\boz]f|)\right\|_{(E_r^{s})_t(\rn)}\\
&\leq C\|b\|_{\Lip_\az(\rn)}\lf[\lf\|\left(M_{\beta \eta}\left(|T_{\Omega,\alpha} f|^{\eta}\right)\right)^{\frac{1}{\eta}}\right\|_{(E_r^{s})_t(\rn)}+\lf\|\left(M_{\alpha\eta+\beta \eta}\left(|f|^{\eta}\right)(x)\right)^{\frac{1}{\eta}}\right\|_{(E_r^{s})_t(\rn)}\right]\\
&\leq C\|b\|_{\Lip_\az(\rn)}\|f\|_{(E_p^{q})_t(\rn)}.
\end{align*}

Next, we will show $(b)$.
Let $K(x):=\frac{\boz(x/|x|)}{|x|^{n-\alpha}}$. Choose $z_0\in\mathbb{R}^n$ such that $|z_0|=3$. For $x\in Q(z_0,2)$, $\frac 1{K(x)}$ be written as the absolutely convergent Fourier series,
$
\frac 1{K(x)}=\sum_{m\in \mathbb{Z}^n}a_me^{2imx}
$
with $\sum_{m\in \mathbb{Z}^n}|a_m|<\infty$. 
For any $x_0\in\rn$ and $r>0$.
Let $Q=Q(x_0,r)$ and $Q_{z_0}=Q(x_0+z_0 r,r)$.
\begin{align*}
\int_Q|b(x)-b_{Q_{z_0}} |dx&=\int_Q(b(x)-b_{Q_{z_0}} )s(x)dx\\
&=|Q_{z_0}|^{-1}\int_Q\int_{Q_{z_0}}s(x)(b(x)-b(y))dydx\\
&=r^{-n}\int_Qs(x)\lf(\int_{Q_{z_0}}(b(x)-b(y))\frac{K(x-y)}{K({x-y})}dy\right)dx,
\end{align*}
where $s(x)=\overline{sgn\lf(\int_{Q_{z_0}}b(x)-b(y)dy\right)}$.
Let \( x \in Q \) and \( y \in Q_{z_0} \), we have \( \frac{y - x}{r} \in Q(z_0, 2) \). Therefore, using the fact that \( [b, T_{\Omega, \alpha}] : (E_p^q)_t(\mathbb{R}^n) \to (E_r^s)_t(\mathbb{R}^n) \), along with H\"older's inequality and the condition \( \frac{\alpha + \beta}{n} = \frac{1}{q} - \frac{1}{s} \), we obtain that
\begin{align*}
\frac{r^{n-\alpha}}{r^{n}}&\int_Qs(x)\lf(\int_{Q_{z_0}}(b(x)-b(y))\frac{K(x-y)}{K(\frac{x-y}r)}dy\right)dx\\
&=r^{-\alpha}\sum_{m\in\mathbb{Z}^n}a_m\int_Qs(x)\lf(\int_{Q_{z_0}}(b(x)-b(y))K(x-y)e^{-2im\delta \cdot \frac{y}{r}}dy\right)e^{2im\delta \cdot \frac{x}{r}}dx\\
&\le r^{-\alpha}\sum_{m\in\mathbb{Z}^n}|a_m|\int_Q\lf|[b,T_{\Omega,\alpha}]f_m(x)g_m(x)\right|dx\\
&\le Cr^{-\alpha}\sum_{m\in\mathbb{Z}^n}|a_m|\lf\|[b,T_{\Omega,\alpha}](f_m)\right\|_{(E_r^s)_t(\rn)}
\lf\|g_m\right\|_{(E_{r'}^{s'})_t(\rn)}\\
&\le Cr^{-\alpha}\sum_{m\in\mathbb{Z}^n}|a_m|\|\chi_{Q_{z_0}}\|_{(E_p^q)_t(\rn)}
\|\chi_Q\|_{(E_{r'}^{s'})_t(\rn)}\\
&\leq C|Q|^{\frac{1}{q}+1-\frac{1}{s}-\frac{\alpha}{n}}\\
&\leq C|Q|^{1+\frac{\beta}{n}}
\end{align*}
Thus, we deduce that
$$
\frac{1}{|Q|^{1+\frac{\beta}{n}}}\int_Q|b(x)-b_Q |dx\le\frac{2}{|Q|^{1+\frac{\beta}{n}}}\int_Q|b(x)-b_{Q_{z_0}} |dx\le C.
$$
Then $b\in \Lip_\beta(\rn)$.

This completes the proof of Theorem \ref{TH1.2}.
\end{proof}

\end{document}